\documentclass[12pt]{article}

 \usepackage{latexsym,amsfonts,amsmath,theorem,amssymb}
 \usepackage{enumerate}
 \usepackage{hyperref}
 
 \usepackage[left=2cm, right=2cm, top=3cm, bottom=3cm]{geometry}

\newcommand{\bsgamma}{\boldsymbol{\gamma}}
\newcommand{\setu}{{\mathfrak{u}}}
\newcommand{\bszero}{\boldsymbol{0}}
\newcommand{\bst}{\boldsymbol{t}}
\newcommand{\bsx}{\boldsymbol{x}}
\newcommand{\bsy}{\boldsymbol{y}}
\newcommand{\rd}{\,\mathrm{d}}
\newcommand{\NN}{\mathbb{N}}
\newcommand{\setv}{{\mathfrak{v}}}
\newcommand{\setU}{{\mathfrak{U}}}
\newcommand{\setw}{{\mathfrak{w}}}

\newtheorem{theorem}{Theorem}
\newtheorem{corollary}{Corollary}
\newtheorem{proposition}{Proposition}

\newtheorem{definition}{Definition}
\newtheorem{remark}{Remark}
\newtheorem{example}{Example}

\newenvironment{proof}{\begin{trivlist}
    \item[\hskip\labelsep{\bf Proof.}]}{$\hfill\Box$\end{trivlist}}

\allowdisplaybreaks

\begin{document}

\title{Truncation Dimension for Function Approximation}

\author{Peter Kritzer\thanks{Supported by the Austrian
Science Fund (FWF) Project F5506-N26.}, Friedrich Pillichshammer\thanks{Supported by the 
Austrian Science Fund (FWF) Project F5509-N26.}, and G.W. Wasilkowski}

\date{}

\maketitle

\abstract{
We consider approximation of functions of $s$ variables, where $s$ is very 
large or infinite, that belong to weighted anchored spaces. We study 
when such functions can be approximated by algorithms designed for 
functions with only very small number ${\rm dim^{trnc}}(\varepsilon)$
of variables. Here $\varepsilon$ is the error demand and we refer
to ${\rm dim^{trnc}}(\varepsilon)$ as
the $\varepsilon$-{\em truncation dimension}. 
We show that for sufficiently fast decaying product 
weights and modest error demand 
(up to about $\varepsilon \approx 10^{-5}$) the truncation dimension is surprisingly very small. 
} 

\section{Introduction}\label{sec:intro}
In this paper, we consider weighted anchored spaces of $s$-variate functions 
with bounded (in $L_p$ norm, $1\le p\le \infty$) mixed partial
derivatives of order one. 
More precisely, the functions being approximated are from the 
Banach space $F_{s,p,\bsgamma}$ whose norm is given by 
\[
\|f\|_{F_{s,p,\bsgamma}}\,=\,
\left(\sum_\setu \gamma_\setu^{-p}\,\int_{[0,1]^{|\setu|}}
  |f^{(\setu)}([\bsx_\setu;\bszero_{-\setu}])|^p
  \rd \bsx_\setu\right)^{1/p}. 
\]
Here,
the summation is with respect to the subsets $\setu$ of
$[s]=\{1,\dots,s\}$ (when $s=\infty$ the summation
is with respect to all finite subsets of $\NN$),
and $f^{(\setu)}([\bsx_\setu;\bszero_{-\setu}])$ denotes the
mixed partial derivatives $\prod_{j\in\setu}\frac{\partial}{\partial x_j}$
of $f$ with values of $x_j$ for $j\notin\setu$ being zero. 
A crucial role is played by 
the weights $\gamma_\setu$, which are non-negative real
numbers that quantify the importance of sets
$\bsx_\setu=(x_j)_{j\in\setu}$ of variables.

We continue our considerations from \cite{KrPiWa}, where we dealt with 
low truncation dimension for numerical integration. 
We are interested in a very large number $s$ of variables including $s=\infty$. 
Similar to \cite{KrPiWa}, by $\varepsilon$-{\em truncation dimension} 
(or {\em truncation dimension} for short) we mean (roughly) the smallest 
number $k$ such that the worst case error (measured in the $L_q$ space) of approximating 
$s$-variate functions $f(x_1,\dots,x_s)$ by 
$f_k=f(x_1,\dots,x_k,0,\dots,0)$ is no greater than the error demand 
$\varepsilon$ (see Definition~\ref{def:td} for more). We denote 
this minimal number $k$ by
${\rm dim^{trnc}}(\varepsilon)$.

Note that if the truncation dimension is small, say ${\rm dim^{trnc}}(\varepsilon)=3$, then the
$s$-variate approximation problem can be replaced by the much easier 
${\rm dim^{trnc}}(\varepsilon)$-variate one, and any efficient algorithm
for dealing with functions of only very few variables becomes also
efficient for functions of $s$ variables. 

The main result of this paper is the observation that the 
$\varepsilon$-truncation dimension is 
surprisingly small for modest error 
demand $\varepsilon$ and the weights decaying sufficiently fast. 
For instance, for product weights
\[
  \gamma_\setu\,=\,\prod_{j\in\setu}j^{-a}
\]
and the parameters $p=q=2$, we have the following upper bounds 
$k(\varepsilon)$ on ${\rm dim^{trnc}}(\varepsilon)$  for $a=3,4,5$: 
\[
\begin{array}{c||c|c|c|c|c|c||l}
\varepsilon &  10^{-1}  &  10^{-2}  & 10^{-3}  & 10^{-4}  & 10^{-5}  &  10^{-6} & \\
\hline
k(\varepsilon)& 2      &  5       &  12      & 31     & 79      & 198
&a=3\\
\hline
k(\varepsilon)& 2      &  3        &  6     & 11     &  22     &  42
&a=4\\
\hline
k(\varepsilon)& 1      &  2        &  4     & 6     &  11     &  18
&a=5
\end{array}
\]

We stress that our definition of truncation dimension is different
from the one proposed in statistical literature. There the dimension
depends on a particular function via its ANOVA decomposition 
which, in general, cannot be computed. Moreover, for functions from 
spaces with ANOVA decomposition, small truncation dimension cannot be 
utilized unless the weights $\gamma_\setu$ are such that the anchored
and ANOVA norms are equivalent. Such equivalence has recently been 
studied in \cite{GnHeHiRiWa16,HeRi13,HeRiWa15,SH,KrPiWa2} and, 
in particular, there is an equivalence independent of $s$ for product 
weights that are summable, with the equivalence constant bounded
from above by 
\[
  \sum_\setu\gamma_\setu<\infty.
\]
For product weights mentioned above we have 
\[
\sum_\setu\gamma_\setu\,=\,\prod_{j=1}^\infty(1+j^{-a})\,<\,\infty
\]
and hence the corresponding efficient algorithms for anchored spaces
can also be used efficiently for ANOVA spaces. 

The paper is structured as follows. In Section 
\ref{defspace}, we define the anchored spaces $F_{s,p,\bsgamma}$, 
and in Section \ref{sec:problem} we outline the problem setting. 
We give our results for anchored spaces in Section \ref{sec:anchored}, 
and discuss examples of special kinds of weights in Sections 
\ref{sec:product} and \ref{sec:pod}.

\section{Basic Concepts}

\subsection{Anchored Spaces}\label{defspace}
In this section, we briefly recall definitions and basic properties
of $\bsgamma$-weighted anchored 
Sobolev spaces of $s$-variate functions. More
detailed information can be found in \cite{HeRiWa15,SH,Was14}.

Here we follow \cite[Section~2]{Was14}: For $p \in [1,\infty]$
let $F=W_{p,0}^1([0,1])$ be the space of functions defined on $[0,1]$
that vanish at zero, are absolutely continuous, and have bounded
derivative in the $L_p$ norm. We endow $F$ with the norm
$\|f\|_F=\|f'\|_{L_p}$ for $f \in F$.

For $s \in \mathbb{N}$ and
\[
  [s]\,:=\,\{1,2,\dots,s\},
\]
we will use $\setu,\setv$ to denote
subsets of $[s]$, i.e.,
\[
  \setu,\setv\subseteq\,[s].
\] 
Moreover, for $\bsx=(x_1,x_2,\ldots,x_s)\in[0,1]^s$ and
$\setu\subseteq [s]$, $[\bsx_\setu;\bszero_{-\setu}]$ denotes the
$s$-dimensional vector with all $x_j$ for $j\notin\setu$ replaced
by zero, i.e.,
\[
  [\bsx_\setu;\bszero_{-\setu}]\,=\,(y_1,y_2,\dots,y_s)\quad
   \mbox{with}\quad  y_j\,=\,\left\{\begin{array}{ll} x_j &
    \mbox{if\ }j\in\setu,\\ 0 & \mbox{if\ } j\notin\setu.\end{array}
    \right.
\]
We also write $\bsx_\setu$ to denote the $|\setu|$-dimensional
vector $(x_j)_{j\in\setu}$ and
\[
  f^{(\setu)}\,=\,\frac{\partial^{|\setu|} f}{\partial \bsx_{\setu}}\,
=\,\prod_{j\in\setu}\frac{\partial}{\partial x_j}\,f
  \quad\mbox{with}\quad f^{(\emptyset)}=f.
\]

For $s \in \mathbb{N}$ and nonempty $\setu \subseteq [s]$ let $F_{\setu}$
be the completion of the space spanned by
$f(\bsx)=\prod_{j \in \setu} f_j(x_j)$ for $f_j \in F$
and $\bsx=(x_1,\ldots,x_s) \in [0,1]^s$, with the norm
$$\|f\|_{F_{\setu}}=\|f^{(\setu)}\|_{L_p}.$$
Note that $F_{\setu}$ is a
space of functions with domain $[0,1]^s$ that depend only on the
variables listed in $\setu$. 
Moreover, for any $f\in F_\setu$ and $\bsx=(x_1,\ldots,x_s)$, $f(\bsx)=0$ if 
$x_j=0$ for some $j\in\setu$. 
For $\setu=\emptyset$, let $F_{\setu}$ be the
space of constant functions with the natural norm.

Consider next a sequence $\bsgamma=(\gamma_\setu)_{\setu\subseteq[s]}$
of non-negative real numbers, called {\em weights}. Since some
weights could be zero, we will use
\[
  \setU\,=\,\{\setu\subseteq [s]\,:\,\gamma_\setu>0\}
\]
to denote the collection of positive weights. For $p\in[1,\infty]$,
we define  
the corresponding weighted {\em anchored} space 
\[
F_{s,p,\bsgamma}\,=\, {\rm span} \left(\bigcup_{\setu \in \setU} F_{\setu}\right)
\]
with the norm 
\[
  \|f\|_{F_{s,p,\bsgamma}}\,=\,\left\{
\begin{array}{ll}
 \left(\sum_{\setu\in\setU}\frac{1}{\gamma_\setu^p}\,
    \|f^{(\setu)}([\cdot_\setu;\bszero_{-\setu}])\|_{L_p}^p
   \right)^{1/p} & \mbox{ if $p < \infty$,}\\[0.8em]
 \max_{\setu\in\setU}\frac1{\gamma_\setu}\,
   \mathrm{ess\sup}_{\bsx_\setu\in[0,1]^{|\setu|}}
   |f^{(\setu)}([\bsx_\setu;\bszero_{-\setu}])| &
   \mbox{ if $p = \infty$.}
\end{array}\right.
\]

\begin{remark}
Some of the results of this paper can be extended to spaces of
functions with countably many variables. In such cases, $[s]=\NN$,
the sets $\setu$ are finite subsets of $\NN$, and
$\bsx=(x_j)_{j\in\NN}$ with $x_j\in[0,1]$. Moreover, the anchored
space is the completion of ${\rm span}\left(\bigcup_\setu F_\setu\right)$
with respect to the norm given above. 
\end{remark}

An important class of weights is provided by product weights
\[
  \gamma_\setu\,=\,\prod_{j\in\setu}\gamma_j
\]
for positive reals $\gamma_j$. When dealing with them, we will assume
without any loss of generality 
that
\[
  \gamma_j\,\ge\,\gamma_{j+1}\,>\,0\quad\mbox{for all\ }j.
\]
Note that for product weights we have
$\setU=2^{[s]}=\left\{\setu: \setu\subseteq [s]\right\}$.

For $p=2$, $F_{s,2,\bsgamma}$ is a reproducing kernel Hilbert
space with kernel
\[
  K(\bsx,\bsy)\,=\,\sum_{\setu\in\setU}\gamma_\setu^{2}\prod_{j\in\setu}
  \min(x_j,y_j),
\] 
for $\bsx=(x_1,\ldots,x_s)$ and analogously for $\bsy$, which for
product weights reduces to
\[
  K(\bsx,\bsy)\,=\,\prod_{j=1}^s\left(1+\gamma_j^2\,\min(x_j,y_j)\right).
\]

\subsection{The Function Approximation Problem}\label{sec:problem}
 
We follow \cite{Was14}. Let $q\in [1,\infty]$. For $\setu\in\setU$, let 
$S_\setu:F_\setu\rightarrow L_q ([0,1]^{\left\vert\setu\right\vert})$
be the embedding operator, 
\[
  S_\setu (f_\setu)=f_\setu\ \ \ \mbox{ for all $f_\setu\in F_\setu$.}
\]
It is well known that 
\[
  \left\Vert S_\setu\right\Vert=\left\Vert S\right\Vert^{|\setu|}
\]
for the space $F_{\setu}$, where $S:F \rightarrow L_q([0,1])$
is the univariate embedding operator. 
 
Let further $\mathcal{L}_q$ be a normed linear space 
such that $F_{s,p,\bsgamma}$ is its subspace and the norm
$\left\Vert\cdot\right\Vert_{\mathcal{L}_q}$ is such that
\[
\left\Vert f_{\setu}\right\Vert_{\mathcal{L}_q}=
\left\Vert f_{\setu}\right\Vert_{L_q
  ([0,1]^{\left\vert\setu\right\vert})} 
\quad\mbox{for all\ }
f_\setu\in F_\setu.
\]
Denote by $S_s$ the embedding operator
\[
S_s: F_{s,p,\bsgamma}\rightarrow \mathcal{L}_q,\quad S_s(f)\,=\,f.
\]

In order to make sure that $S_s$ is continuous, we assume from now on that 
\[
\sum_{\setu \in\setU} \left(\left(\frac{q}{p^{\ast}}
+1\right)^{-\left\vert\setu\right\vert/q}\gamma_{\setu}\right)^{p^{\ast}}<\infty,
\]
where here and throughout this paper $p^*$ denotes the conjugate
of $p$, i.e., $1/p+1/p^*=1$. To see that this condition indeed
ensures continuity of $S_s$, 
we recall the following proposition
from \cite{Was14}. 

\begin{proposition}\label{propgreg}
We have
\[\left\Vert S\right\Vert\le \left(\frac{q}{p^*} +1\right)^{-1/q}
\quad\mbox{and}\quad
\left\Vert S_s\right\Vert\le \left[\sum_{\setu\in\setU}
  \left(\left(\frac{q}{p^{\ast}}+1\right)^{-\left\vert\setu\right\vert/q}
  \,\gamma_{\setu}\right)^{p^*}\right]^{1/p^{\ast}}.
\]
\end{proposition}

Note that for product weights we have
\[
\sum_{\setu\in\setU} \left(\left(\frac{q}{p^{\ast}}
+1\right)^{-\left\vert\setu\right\vert/q}\gamma_{\setu}\right)^{p^*}\, =\,
\prod_{j=1}^s\left(1+\gamma_j^{p^*} \left(\frac{q}{p^*}
+1\right)^{-p^*/q}\right).
\]

We are interested in algorithms for approximating 
$f\in F_{s,p,\bsgamma}$. In this paper we only consider algorithms 
that use so-called
{\it standard information}, i.e., which only use function
evaluations as the allowed information class, and which have the form
\[
A_{s}(f) = \phi(f(\bsx_1), \ldots, f(\bsx_n))
\]
for $\bsx_j \in [0,1]^s$, where
$\phi:\mathbb{R}^n \rightarrow \mathcal{L}_q$.
An important class of algorithms is
provided by the class of linear algorithms which are of the form
$$A_{s}(f) = \sum_{j=1}^n f(\bsx_j)  g_j$$ 
for $g_j \in F_{s,p,\bsgamma}$.
We study the worst case setting, where the error of an algorithm
$A_{s}$ is the operator norm of $S_s-A_{s}$, i.e.,
\[
e(A_{s};F_{s,p,\bsgamma})\,:=\,\|S_s-A_{s}\|\,=\,
\sup_{\|f\|_{F_{s,p,\bsgamma} \le 1}}\|f-A_{s}(f)\|_{\mathcal{L}_q}.
\]

\section{Anchored Decomposition and Truncation Dimension}\label{sec:anchored}
It is well known, see, e.g., \cite{KSWW10a}, that any $f\in
F_{s,p,\bsgamma}$ has the unique {\em anchored decomposition}
\begin{equation}\label{dec-a}
  f\,=\,\sum_{\setu\in\setU} f_\setu,
\end{equation}
where $f_\setu$ is an element of $F_\setu$,
depends only on $x_j$ for $j\in\setu$, and
\begin{equation}\label{anc}
   f_\setu(\bsx)\,=\,0\quad\mbox{if\ }x_j=0\ \mbox{\ for some\ }\ j\in\setu.
\end{equation}
For the empty set $\setu$, $f_\emptyset$ is a constant function.
We stress that in general we do not know what the elements
$f_\setu$ are and algorithms are only allowed to evaluate
the original function $f$.

The anchored  decomposition has the following important
properties, see, e.g.,~\cite{HeRiWa15}:
\begin{equation}\label{der}
   f^{(\setu)}([\cdot_\setu;\bszero_{-\setu}])\,\equiv\,f_{\setu}^{(\setu)},
\end{equation}
\[
f_\setu\,\equiv\,0\quad\mbox{iff}\quad
   f^{(\setu)}([\cdot_\setu;\bszero_{-\setu}])\,\equiv\,0,
\]
\[
   \|f\|_{F_{s,p,\bsgamma}}\, = \, \left\| \sum_{\setu\in\setU}
     f_\setu \right\|_{F_{s,p,\bsgamma}} \,=\,
  \left(\sum_{\setu\in\setU}\gamma_\setu^{-p}\,
    \|f_\setu^{(\setu)}\|_{L_p}^p\right)^{1/p}\quad\mbox{for\ }
    p<\infty, 
\]
and
\[
  \|f\|_{F_{s,\infty,\bsgamma}}\,=\,\max_{\setu\in\setU}
   \frac{\|f_\setu^{(\setu)}\|_{L_\infty}}{\gamma_\setu}
   \quad\mbox{for\ }p=\infty.
\]
For any $\setu\not=\emptyset$, there exists (unique in the $L_p$-sense)
$g \in L_p([0,1]^{|\setu|})$ such that
\[
f_\setu(\bsx)\,=\,\int_{[0,1]^{|\setu|}}g(\bst)\,
\prod_{j\in\setu} 1_{[0,x_j)}(t_j)
  \rd \bst\quad\mbox{and}\quad
  f_\setu^{(\setu)}\,=\,g,
\]
where $1_J(t)$ is the characteristic function of the set $J$,
i.e., $1_J(t)=1$ if $t \in J$ and 0 otherwise.

Moreover, for any $\setu$,
\[
  f([\cdot_\setu;\bszero_{-\setu}])\,=\,\sum_{\setv\subseteq\setu}f_\setv.
\]
In particular, for $k<s$ we have
\begin{equation}\label{trnk}
  f([\bsx_{[k]};\bszero_{-[k]}])\,=\,f(x_1,\dots,x_k,0,\dots,0)
  \,=\,\sum_{\setv\subseteq[k]}f_\setv(\bsx)
\end{equation}
which allows us to compute samples and approximate the
{\em truncated} function
\[
  f_k(x_1,\dots,x_k)\,:=\,\sum_{\setv\subseteq[k]}f_\setv(\bsx).
\]
Moreover, $f_k\in F_{k,p,\bsgamma}\subseteq F_{s,p,\bsgamma}$ and
\[
\|f_k\|_{F_{k,p,\bsgamma}}\,=\,
\|f([\cdot_{[k]};\bszero_{-[k]}])\|_{F_{s,p,\bsgamma}}\,=\,
\left\|\sum_{\setu\subseteq[k]}f_\setu
   \right\|_{F_{s,p,\bsgamma}}.
\]
This leads to the following concept.

\begin{definition}\rm \label{def:td}
For a given error demand $\varepsilon>0$, by
$\varepsilon$-{\em truncation dimension for the approximation problem}
(or {\em truncation dimension} for short), denoted by
${\rm dim^{trnc}}(\varepsilon)$,
we mean the smallest integer $k$ such that 
\[
\left\|\sum_{\setu\not\subseteq[k]}f_\setu\right\|_{\mathcal{L}_q}\,
\le\,\varepsilon\,\left\|\sum_{\setu\not\subseteq[k]}f_\setu
\right\|_{F_{s,p,\bsgamma}}\quad\mbox{for all\ }f\,=\,\sum_{\setu\in\setU}f_\setu\,\in\,F_{s,p,\bsgamma}. 
\]   
\end{definition}

We have the following upper bound on the truncation dimension.

\begin{theorem}\label{thm:dim}
We have
\[
  {\rm dim^{trnc}}(\varepsilon)\,\le\,
  \min\left\{k\,:\,\left(\sum_{\setu\not\subseteq[k]}
  \frac{\gamma_\setu^{p^*}}{(\frac{q}{p^*}+1)^{p^*\,|\setu|/q}}
  \right)^{1/p^*}\,\le\,\varepsilon\right\}\quad\mbox{for\ }p>1,
\]
and 
\[
{\rm dim^{trnc}}(\varepsilon)\,=\,
\min\left\{k\,:\,\max_{\setu\not\subseteq[k]}\gamma_\setu\,\le\,\varepsilon \right\}\quad\mbox{for\ }p=1.
\] 
Here $\sum_{\setu\not\subseteq[k]}$ means summation over all $\setu
\subseteq [s]$ with $\setu\not\subseteq[k]$ and similarly for
$\max_{\setu\not\subseteq[k]}$.
\end{theorem}
The proof of the theorem is in the proof of the next, Theorem \ref{thmtrunc}.

For given $k<s$, let $A_{k}$ be an algorithm for approximating
functions from the space $F_{k,p,\bsgamma}$. We use it to
define the following approximation algorithms for the original space
$F_{s,p,\bsgamma}$,
\begin{equation}\label{trnk-q}
  A^{\rm trnc}_{s,k}(f)\,=\,A_{k}(f([\cdot_{[k]};\bszero_{-[k]}]))
  \,=\,A_{k}(f_k).
\end{equation}
Clearly, the algorithms $A^{\rm trnc}_{s,k}$ are well defined.

We have the following result.

\begin{theorem}\label{thmtrunc}
For every  $k<s$, 
the worst case error of $A^{\rm trnc}_{s,k}$ is bounded  by
\[
  e(A^{\rm trnc}_{s,k};F_{s,p,\bsgamma})\,\le\,
  \left( [e(A_{k};F_{k,p,\bsgamma})]^{p^*}  +
  \sum_{\setu\not\subseteq[k]}
\frac{\gamma_\setu^{p^*}\,}{(\frac{q}{p^*}+1)^{p^* |\setu|/q}} \right)^{1/p^*}
\quad\mbox{for\ }p>1
\]
and by
\[
      e(A^{\rm trnc}_{s,k};F_{s,1,\bsgamma})\,\le\,
   \max\left(e(A_{k};F_{k,1,\bsgamma})\,,\,
     \max_{\setu\not\subseteq[k]}\gamma_\setu\right) \quad\mbox{for\ }p=1.
\]

Moreover, if $k\ge{\rm dim^{trnc}}(\varepsilon)$ then
\[
e(A_{s,k}^{\rm trnc};F_{s,p,\bsgamma})\,\le\,
\left([e(A_{k};F_{k,p,\bsgamma})]^{p^*}+\varepsilon^{p^*}
\right)^{1/p^*}.
\]
\end{theorem}

\begin{proof}
We prove the theorem for $p>1$ only since the proof for $p=1$ is
very similar. Let us first assume that $1<p<\infty$. For any
$f\in F_{s,p,\bsgamma}$ it holds that
\begin{eqnarray*}
   \left\Vert S_s(f)-A^{\rm trnc}_{s,k}(f)\right\Vert_{\mathcal{L}_q}&=&
   \left\Vert S_s(f)-A_{k}(f_k)\right\Vert_{\mathcal{L}_q}\\
   &=&\left\Vert S_k(f_k) -A_{k}(f_k) +
   \sum_{\setu\not\subseteq[k]} S_{\setu}(f_\setu)\right\Vert_{\mathcal{L}_q}\\
  &\le& e(A_{k};F_{k,p,\bsgamma})\,\|f_k\|_{F_{k,p,\bsgamma}}+
   \sum_{\setu\not\subseteq[k]} \left\Vert S_{\setu} (f_\setu)
   \right\Vert_{\mathcal{L}_q}\\
   &\le& e(A_{k};F_{k,p,\bsgamma})\,\|f_k\|_{F_{k,p,\bsgamma}}+
   \sum_{\setu\not\subseteq[k]} \left\Vert
     f_\setu\right\Vert_{F_\setu}
  \left\Vert S_\setu\right\Vert.
\end{eqnarray*}
We now have
\begin{eqnarray*}
 \sum_{\setu\not\subseteq[k]} \left\Vert
   f_\setu\right\Vert_{F_\setu}\left\Vert S_\setu\right\Vert&=&
 \sum_{\setu\not\subseteq[k]} \gamma_\setu^{-1}\left\Vert f_\setu\right\Vert_{F_\setu} \gamma_\setu \left\Vert S_\setu\right\Vert\\
 &\le & \left[\sum_{\setu\not\subseteq[k]} \left(\frac{\left\Vert f_\setu\right\Vert_{F_\setu}}{\gamma_\setu}\right)^p\right]^{1/p}
 \left[\sum_{\setu\not\subseteq[k]}(\gamma_\setu \left\Vert S_\setu\right\Vert)^{p^*}\right]^{1/p^*}\\
 &\le & \left[\sum_{\setu\not\subseteq[k]} \left(\frac{\left\Vert f_\setu\right\Vert_{F_\setu}}{\gamma_\setu}\right)^p\right]^{1/p}
 \left[\sum_{\setu\not\subseteq[k]}\left(\gamma_\setu \left(\frac{q}{p^*}+1\right)^{-\left\vert\setu\right\vert/q}\right)^{p^*}\right]^{1/p^*},
\end{eqnarray*}
where we used Proposition \ref{propgreg} in the last step. 

Hence, putting together, we get
\begin{eqnarray*}
  \left\Vert S_s(f)-A^{\rm trnc}_{s,k}(f)\right\Vert_{\mathcal{L}_q}&\le&
  e(A_{k};F_{k,p,\bsgamma})\,\left(\sum_{\setu\subseteq[k]}\gamma_\setu^{-p}\,
   \|f_\setu^{(\setu)}\|_{L_p}^p\right)^{1/p}\\
  &&\quad +
   \left(\sum_{\setu\not\subseteq[k]}\frac{\gamma_\setu^{p^*}}{(\frac{q}{p^*}+1)^{p^* |\setu|/q}}
   \right)^{1/p^*}\,\left(\sum_{\setu\not\subseteq[k]}\gamma_\setu^{-p}\,
    \|f_\setu^{(\setu)}\|_{L_p}^p\right)^{1/p}.\\
\end{eqnarray*}
Using the H\"{o}lder inequality once more, we obtain
\begin{eqnarray*}
\lefteqn{\left\Vert S_s(f)-A^{\rm trnc}_{s,k}(f)\right\Vert_{\mathcal{L}_q}}\\
  &\le& \left(\sum_{\setu\subseteq[s]}\gamma_\setu^{-p}\,
   \|f_\setu^{(\setu)}\|_{L_p}^p\right)^{1/p}\,
   \left([e(A_{k};F_{k,p,\bsgamma})]^{p^*}+
   \sum_{\setu\not\subseteq[k]}\frac{\gamma_\setu^{p^*}}{(\frac{q}{p^*}+1)^{p^* |\setu|/q}}
   \right)^{1/p^*}\\
  & = & \|f\|_{F_{s,p,\bsgamma}}\,
   \left([e(A_{k};F_{k,p,\bsgamma})]^{p^*}+
   \sum_{\setu\not\subseteq[k]}\frac{\gamma_\setu^{p^*}}{(\frac{q}{p^*}+1)^{p^* |\setu|/q}}
   \right)^{1/p^*}.
\end{eqnarray*}
This shows the result for $1<p<\infty$. For $p=\infty$, the result is obtained by letting $p \rightarrow \infty$. This completes the proof.
\end{proof}

We now apply this theorem to two important classes of weights: product weights and product order-dependent weights.

\subsection{Product Weights}\label{sec:product}
We assume in this section that the weights have the following
{\em product} form
\[
\gamma_\setu\,=\,\prod_{j\in\setu}\gamma_j\quad\mbox{for}\quad
 1 \ge \gamma_j\,\ge\,\gamma_{j+1}\,>\,0,
\] 
introduced in \cite{SlWo98}. 
Here the empty product is considered to be 1, i.e.,
$\gamma_{\emptyset}=1$. As already mentioned, for product weights we always
have $\setU=2^{[s]}$.

\begin{proposition}\label{proptrunc}
For product weights and $k<s$,  the truncation error is bounded by
\begin{eqnarray*} 
\left(\sum_{\setu\not\subseteq[k]}
\frac{\gamma_\setu^{p^*}}{(\frac{q}{p^*}+1)^{p^* |\setu|/q}}
\right)^{1/p^*}  & \le &  \prod_{j=1}^s\left(1+\frac{\gamma_j^{p^*}}{(\frac{q}{p^*}+1)^{p^*/q}}\right)^{1/p^*} \\
& & \times \left(1-\exp\left(\frac{-1}{(\frac{q}{p^*}+1)^{p^*/q}}\,
\sum_{j=k+1}^s \gamma_j^{p^*}
\right)\right)^{1/p^*}
\end{eqnarray*}
for $p>1$, and it is equal to
\[
  \max_{\setu\not\subseteq[k]}\gamma_\setu \quad\mbox{for\ }p=1.
\]
\end{proposition}

\begin{proof}
The proof for $p=1$ is trivial. For $p>1$, we have
\begin{eqnarray*}
\lefteqn{\sum_{\setu\not\subseteq[k]}\frac{\gamma_\setu^{p^*}}{(\frac{q}{p^*}+1)^{p^* |\setu|/q}}
   =\sum_{\setu\subseteq [s]}\frac{\gamma_\setu^{p^*}}{(\frac{q}{p^*}+1)^{p^* |\setu|/q}}
   -\sum_{\setu\subseteq[k]}\frac{\gamma_\setu^{p^*}}{(\frac{q}{p^*}+1)^{p^* |\setu|/q}}}\\
   &=&\prod_{j=1}^s\left(1+\frac{\gamma_j^{p^*}}{(\frac{q}{p^*}+1)^{p^*/q}}\right)
   -\prod_{j=1}^k\left(1+\frac{\gamma_j^{p^*}}{(\frac{q}{p^*}+1)^{p^*/q}}\right)\\
   &=&\prod_{j=1}^s\left(1+\frac{\gamma_j^{p^*}}{(\frac{q}{p^*}+1)^{p^*/q}}\right)
   \left(1-\prod_{j=k+1}^s \left(1+\frac{\gamma_j^{p^*}}{(\frac{q}{p^*}+1)^{p^*/q}}\right)^{-1}\right).
\end{eqnarray*}
We have
\begin{eqnarray*}
1-\prod_{j=k+1}^s\left(1+\frac{\gamma_j^{p^*}}{(\frac{q}{p^*}+1)^{p^*/q}}
\right)^{-1}
 &=&1-\exp\left(-\sum_{j=k+1}^s\log\left(1+\frac{\gamma_j^{p^*}}{(\frac{q}{p^*}+1)^{p^*/q}}\right)\right)\\
 &\le&1-\exp\left(-\sum_{j=k+1}^s \frac{\gamma_j^{p^*}}{(\frac{q}{p^*}+1)^{p^*/q}}\right),
\end{eqnarray*}
where $\log$ denotes the natural logarithm, and the last inequality is due to $\log (1+x)\le x$ for all $x >-1$. This completes 
the proof. 
\end{proof}

We have the following corollaries:

\begin{corollary}
Consider product weights. Then ${\rm dim^{trnc}}(\varepsilon)$ is
bounded from above by   
\[
  \min\left\{k\,:\,
1-\exp\left(\frac{-1}{(\frac{q}{p^*}+1)^{p^*/q}}\sum_{j=k+1}^s
\gamma_j^{p^*}\right)\le\frac{\varepsilon^{p^*}}{\prod_{j=1}^s
  \left(1+\frac{\gamma_j^{p^*}}{(\frac{q}{p^*}+1)^{p^*/q}}\right)}
  \right\}
\]
for $p>1$, and is equal to 
\[
  \min\left\{k\,:\,\max_{\setu\not
  \subseteq[s]}\gamma_\setu\le\varepsilon\right\}
\]
for $p=1$. 
\end{corollary}

\begin{corollary}\label{cortrunc}
  Consider product weights  and $k<s$.

Then the error $e(A^{\rm trnc}_{s,k};F_{s,p,\bsgamma})$ is bounded from
above by
\begin{eqnarray*}
&&  \Bigg([e(A_{k};F_{k,p,\bsgamma})]^{p^*}\\
  && +\prod_{j=1}^s\left(1+\frac{\gamma_j^{p^*}}{(\frac{q}{p^*}+1)^{p^*/q}}\right)\,
\left(1-\exp
\left(\frac{-1}{(\tfrac{q}{p^*}+1)^{p^*/q}}\sum_{j=k+1}^s 
\gamma_j^{p^*}\right)\right)
\Bigg)^{1/p^*}
\end{eqnarray*}
for $p>1$, and by 
\[
  \max\left(e(A_{k};F_{k,1,\bsgamma})\,,\,
\max_{\setu\not\subseteq[k]}\gamma_\setu
\right)  
\]
for $p=1$. 
Note that if $\gamma_j\le1$ for all $j$ then 
$\max_{\setu\not\subseteq[s]}\gamma_\setu\,=\,\gamma_{k+1}$. 

Therefore, for the worst case error of $A_{s,k}^{\rm trnc}$ not to
exceed the error demand $\varepsilon>0$, it is enough to choose 
$k=k(\varepsilon)$
so that
\begin{equation}\label{trunccrit}
  1-\exp\left(\frac{-1}{(\frac{q}{p^*}+1)^{p^*/q}}\,
  \sum_{j=k+1}^s \gamma_j^{p^*}\right)\,\le\,
  \frac{1}{2} \varepsilon^{p^*} \prod_{j=1}^s\left(1+\frac{\gamma_j^{p^*}}{(\frac{q}{p^*}+1)^{p^*/q}}\right)^{-1} ,
\end{equation}
(or $\gamma_{k+1}\le\varepsilon$
for $p=1$),
and next to choose $n=n(\varepsilon)$ so that
\[
 e(A_{k};F_{k,p,\bsgamma})\,\le\,\frac{\varepsilon}{2^{1/p^*}}.
\]
\end{corollary}

Clearly the inequality \eqref{trunccrit} for $p>1$ is equivalent to
\begin{equation}\label{star}
  \sum_{j=k+1}^s \gamma_j^{p^*}\,\le\,-\left(\frac{q}{p^*}+1\right)^{p^*/q}\,\log\left(1-
 \frac{1}{2} \varepsilon^{p^*} \prod_{j=1}^s\left(1+\frac{\gamma_j^{p^*}}{(\frac{q}{p^*}+1)^{p^*/q}}\right)^{-1}\right).
\end{equation}

\begin{example}\rm \label{ex:1}
Consider large $s$ including $s=\infty$ and
\[
\gamma_\setu\,=\,\prod_{j\in\setu}j^{-a}\quad\mbox{for\ }
a\,>\,1/p^*.
\]

For $p=1$, we have
\[
{\rm dim^{trnc}}(\varepsilon)\,=\,\left\lceil 
\varepsilon^{-1/a}-1\right\rceil.
\]
In particular we have  
\[
\begin{array}{c||c|c|c|c|c|c||l}
\varepsilon &  10^{-1}  &  10^{-2}  & 10^{-3}  & 10^{-4}  & 10^{-5}  &  10^{-6} & \\
\hline
{\rm dim^{trnc}}(\varepsilon)& 3  &    9  & 31 & 99 & 316  &  999
&a=2\\
\hline
{\rm dim^{trnc}}(\varepsilon)& 2   &  4   &  9   & 21     &  46      &  99
&a=3\\
\hline
{\rm dim^{trnc}}(\varepsilon)& 1   &  3   &  5   & 9     &   17      &  31
&a=4\\
\hline
{\rm dim^{trnc}}(\varepsilon)& 1   &  2   &  3   & 6     &  9        &  15
&a=5
\end{array}
\] 

\medskip 
Consider next $p>1$.
Unlike in the case $p=1$, we do not know the exact values of the truncation
dimension. However, we have its upper bounds $k(\varepsilon)$ that are
small,
\[
  {\rm dim^{trnc}}(\varepsilon)\,\le\,k(\varepsilon).
\]
We use the estimate
\[
  \frac{(k+1)^{-ap^*+1}}{ap^*-1}\,=\,\int_{k+1}^\infty x^{-ap^*}\rd x\,<\,
\sum_{j=k+1}^\infty j^{-ap^*}\,\le\,\int_{k+1/2}^\infty x^{-ap^*}\rd x\,=\,
\frac{(k+1/2)^{-ap^*+1}}{a\,p^*-1}.
\]
Note that the relative error when using the upper bound to approximate
the sum is bounded by 
\[
\frac{a\,p^*-1}{2k+1}+O(k^{-2})
\]
and is small for large $k$.
To satisfy \eqref{star}, it is enough to take
$k=k(\varepsilon)$ given by 
\[
k\,=\,   
\left\lceil\left(\frac{-(\tfrac{q}{p^*}+1)^{-p^*/q}\,(ap^*-1)^{-1}}
{
  \log\left(1-\frac{\varepsilon^{p^*}}{2}  \prod_{j=1}^s\left(1+\frac{j^{-a p^*}}{(\frac{q}{p^*}+1)^{p^*/q}}\right)^{-1}\right)}
\right)^{1/(ap^*-1)}-\frac12\right\rceil.
\]

For $p=p^*=2$, which corresponds to the classical Hilbert space
setting, we have
\begin{equation}\label{eqkeps}
k(\varepsilon)\,=\,
\left\lceil \left(\frac{-(\tfrac{q}{2}+1)^{-2/q}\,(2\,a-1)^{-1}}{
\log\left(1-\frac{\varepsilon^2}{2}  \prod_{j=1}^s\left(1+\frac{j^{-2a}}{(\frac{q}{2}+1)^{2/q}}\right)^{-1}\right)}\right)^{1/(2a-1)}
-\tfrac12\right\rceil.
\end{equation} 
If also $q=2$, then 
\[
k(\varepsilon)\,=\,\left\lceil \left(-2\,(2\,a-1)\,
\log\left(1-\frac{\varepsilon^2}{2}  \prod_{j=1}^s\left(1+\frac{j^{-2a}}{2}\right)^{-1}\right)\right)^{-1/(2a-1)}
-\tfrac12\right\rceil.
\]

Since $s$ could  be huge or $s=\infty$, 
in calculating the values of $k(\varepsilon)$, we slightly overestimated the
product $\prod_{j=1}^s\left(1+\frac{j^{-2a}}{2}\right)$ in the following way:
\begin{eqnarray*}
  \prod_{j=1}^s\left(1+\frac{j^{-2a}}{2}\right) &\le& \prod_{j=1}^\infty \left(1+\frac{j^{-2a}}2\right)
   \,\le\,\prod_{j=1}^{1000}\left(1+\frac{j^{-2a}}2\right)\,
   \exp\left(\sum_{j=1001}^\infty\frac{j^{-2a}}2\right)\\
  &\le&\prod_{j=1}^{1000}\left(1+\frac{j^{-2a}}2\right)\,
   \exp\left(\frac12\,\int_{1000.5}^\infty x^{-2a}\rd x\right)\\
  &=&   \prod_{j=1}^{1000}\left(1+\frac{j^{-2a}}2\right)\,
  \exp\left(\frac1{2\,(2a-1)}\,1000.5^{-2a+1}\right).
\end{eqnarray*}
This gave us the following estimations for $\prod_{j=1}^s\left(1+\frac{j^{-2a}}{2}\right)$
for $p=2$:
\[
   1.56225\ \mbox{for\ }a=2,\qquad 1.51302 \ \mbox{for\ }a=3, \qquad
1.50306\ \mbox{for\ }a=4, \qquad
1.50075\ \mbox{for\ }a=5.
\] 
Below we give values of $k(\varepsilon)$ for $a=2,3,4,5$ and $p=q=2$ 
using the estimates above.
\[
\begin{array}{c||c|c|c|c|c|c||l}
\varepsilon &  10^{-1}  &  10^{-2}  & 10^{-3}  & 10^{-4}  & 10^{-5}  &  10^{-6} & \\
\hline
k(\varepsilon)& 4      &  17       & 80      & 373      & 1733    & 8045
&a=2\\
\hline
k(\varepsilon)& 2      &  5       &  12      & 31     & 79      & 198
&a=3\\
\hline
k(\varepsilon)& 2      &  3        &  6     & 11     &  22     &  42
&a=4\\
\hline
k(\varepsilon)& 1      &  2        &  4     & 6     &  11     &  18
&a=5
\end{array}
\]
It is clear that $k(\varepsilon)$ decreases with increasing $a$. 
To check whether the estimates above are sharp, we also calculated 
$k(\varepsilon)$ for $s=1000000$ directly by computing
\[
\prod_{j=1}^s\left(1+\frac{j^{-2a}}2\right)-
\prod_{j=1}^k\left(1+\frac{j^{-2a}}2\right)
\]
and choosing the smallest $k$ for which the difference above is
not greater than $\varepsilon^2/2$. The values of $k(\varepsilon)$
obtained this way are exactly the same.

We now consider $p=\infty$ and $q=2$. By computing 
\[
\prod_{j=1}^s\left(1+\frac{j^{-a}}{\sqrt{3}}\right)-
\prod_{j=1}^k\left(1+\frac{j^{-a}}{\sqrt{3}}\right)
\]
we obtained the following values of $k(\varepsilon)$ for $s=1000000$, which is the smallest $k$ for which the difference above is not greater than $\varepsilon/2$.
\[
\begin{array}{c||c|c|c|c|c|c||l}
\varepsilon &  10^{-1}  &  10^{-2}  & 10^{-3}  & 10^{-4}  & 10^{-5}  &  10^{-6} & \\
\hline
k(\varepsilon)& 3      &  10        &  32    & 101     &  319      & 1010
&a=3\\
\hline
k(\varepsilon)& 2      &  4        &  9     & 19    &   40      & 86
&a=4\\
\hline
k(\varepsilon)& 1      &  3        &  5     & 8     &   15      & 26
&a=5
\end{array}
\]

We do not present the values of $k(\varepsilon)$ for $a=2$ since they 
are too large to be of practical interest. For instance 
$k(10^{-3})=2069$ and $k(10^{-4})=7230$. 

For some particular values of $q$, the norm of the embedding operator 
$S$ is known, as for example for $q=1$, in which case it equals 
$(1+p^*)^{-1/p^*}$.

Let us now consider the case $p=p^*=2$, $s=1000000$, and $q=1$. In 
this case we obtain from \eqref{eqkeps}: 
\[
\begin{array}{c||c|c|c|c|c|c||l}
\varepsilon &  10^{-1}  &  10^{-2}  & 10^{-3}  & 10^{-4}  & 10^{-5}  &  10^{-6} & \\
\hline
k(\varepsilon)& 4      &  16       & 76      & 354      & 1643    & 7628
&a=2\\
\hline
k(\varepsilon)& 2      &  5       &  12      & 30     & 76      & 192
&a=3\\
\hline
k(\varepsilon)& 2      &  3        &  6     & 11     &  21     &  41
&a=4\\
\hline
k(\varepsilon)& 1      &  2        &  4     & 6     &  10     &  17
&a=5
\end{array}
\]
For comparison, we also consider the values of $k(\varepsilon)$, by using the precise formula for 
$\| S_\setu\|=\| S \|^{|\setu|}=(1+p^*)^{-|\setu |/p^*}$ in the proof of Theorem \ref{thmtrunc}. This yields, 
instead of \eqref{eqkeps}:
\begin{equation}\label{eqkepsexact}
k(\varepsilon)\,=\,\left\lceil \left(-3\,(2\,a-1)\,
\log\left(1-\frac{\varepsilon^2}{2}  \prod_{j=1}^s\left(1+\frac{j^{-2a}}{3}\right)^{-1}\right)\right)^{-1/(2a-1)}-\tfrac12\right\rceil.
\end{equation} 
Then we obtain from \eqref{eqkepsexact}:
\[
\begin{array}{c||c|c|c|c|c|c||l}
\varepsilon &  10^{-1}  &  10^{-2}  & 10^{-3}  & 10^{-4}  & 10^{-5}  &  10^{-6} & \\
\hline
k(\varepsilon)& 3      &  14       & 67      & 312      & 1449   & 6727
&a=2\\
\hline
k(\varepsilon)& 2      &  4       &  11      & 28     & 71      & 178
&a=3\\
\hline
k(\varepsilon)& 1      &  3        &  5     & 10     &  20     &  39
&a=4\\
\hline
k(\varepsilon)& 1      &  2        &  4     & 6     &  10     &  17
&a=5
\end{array}
\]
We see that the values of $k(\varepsilon)$ computed 
using the precise value of $\| S\|$ are lower than 
our general bounds, 
but not too much. 
\end{example}

\subsection{Product Order-Dependent Weights}\label{sec:pod}
We assume in this section that the weights have the following 
{\em product order-dependent (POD)} form, 
\[
\gamma_{\setu}\,=\,c_1\,(|\setu|!)^{b}\,\prod_{j\in\setu}\gamma_j,
\]
introduced in \cite{KSS12}. Here $c_1$ is a positive constant. 
Since the truncation error for $p=1$ is 
$\max_{\setu\not\subseteq[k]}\gamma_\setu$, we restrict the
attention in this section to $p>1$, i.e., $p^*<\infty$.
We will use $[k+1:s]$ to denote 
\[
   [k+1:s]\,=\,\{k+1,k+2,\dots,s\} \quad\mbox{or}\quad
   \{k+1,k+2,\dots\}\ \mbox{\ if\ }\ s\,=\,\infty.
\]

\begin{proposition} For POD weights and $k< s$, the truncation error is
bounded by
\[
\left(\sum_{\setu\not\subseteq[k]}\frac{\gamma_\setu^{p^*}}
{(\tfrac{q}{p^*}+1)^{p^* |\setu|/q}}\right)^{1/p^*}\,\le\, \left(\sum_{\setv\subseteq[k]}\frac{\gamma_\setv^{p^*}}{(\tfrac{q}{p^*}+1)^{p^* |\setv|/q }} \right)^{1/p^*}\,T(k),
\]   
where
\[T(k)\,=\,\left(\sum_{l=1}^{s-k}\left(\frac{(l+k)!}{k!}\right)^{b\,p^*}\,
\frac1{(\tfrac{q}{p^*}+1)^{p^* l/q}}\sum_{\substack{\setw\subseteq[k+1:s] \\ |\setw|=l}}
\prod_{j\in\setw}\gamma_j^{p^*}\right)^{1/p^*}.
\]
\end{proposition}

\begin{proof}   
Of course we have
\begin{eqnarray*}
\sum_{\setu\not\subseteq[k]}\frac{\gamma_\setu^{p^*}}
{(\tfrac{q}{p^*}+1)^{p^* |\setu|/q}}&=&\sum_{\setv\subseteq[k]}
\sum_{\emptyset\not=\setw\subseteq[k+1:s]}
\frac{\gamma_{\setv\cup\setw}^{p^*}}{(\tfrac{q}{p^*}+1)^{p^*(|\setv|+|\setw|)/q}}\\
&=&\sum_{\setv\subseteq[k]}\frac{\gamma_\setv^{p^*}}{(\tfrac{q}{p^*}+1)^{p^* |\setv|/q }}
\,T(\setv,k)^{p^*},
\end{eqnarray*}
where
\begin{eqnarray*}T(\setv,k)^{p^*}&=&
\sum_{\emptyset\not=\setw\subseteq[k+1:s]}\left(\frac{(|\setv|+|\setw|)!}
{|\setv|!}\right)^{b\,p^*}\,\frac1{(\tfrac{q}{p^*}+1)^{p^* |\setw|/q}}\,
\prod_{j\in\setw}\gamma_j^{p^*}\\
&=&\sum_{l=1}^{s-k}\left(\frac{(|\setv|+l)!}{|\setv|!}\right)^{b\,p^*}\,
\frac1{(\tfrac{q}{p^*}+1)^{p^* l/q}}\,\sum_{\substack{\setw\subseteq[k+1:s]\\ |\setw|=l}}
\prod_{j\in\setw}\gamma_j^{p^*}.
\end{eqnarray*}
Since $|\setv|\le k$, we have
\[\frac{(|\setv|+l)!}{|\setv|!}\,\le\,\frac{(k+l)!}{k!}.
\]
This completes the proof. 
\end{proof}

\begin{example}\rm Consider large $s$ 
and
\[
  \gamma_j\,=\,\frac{c_2}{j^a}\quad\mbox{for\ }a\,>\,\max(1/p^*,b).
\]
Clearly
\begin{eqnarray*}
\sum_{\substack{\setw\subseteq[k+1:s]\\ |\setw|=l}}
\prod_{j\in\setw}\gamma_j^{p^*}
 &=& c_2^{l p^*}\,\sum_{j_1=k+1}^s j_1^{-a\,p^*}\,\sum_{j_2=j_1+1}^s j_2^{-a\,p^*}
\ldots\sum_{j_l=j_{l-1}+1}^s j_l^{-a\,p^*}\\
&\le& c_2^{l p^*}\,\int_{k+1/2}^\infty x_1^{-a\,p^*}\,
\int_{x_1}^\infty x_2^{-a\,p^*}\,
\ldots\int_{x_{l-1}}^\infty x_l^{-a\,p^*} \rd x_l\ldots\rd x_2\rd x_1 \\
&=&\left(\frac{c_2^{p^*}}{(k+1/2)^{a\,p^*-1}\,(a\,p^*-1)}\right)^l\,\frac1{l!}.
\end{eqnarray*}
Therefore
\begin{equation}\label{pod-res}
  T(k)\,\le\,\,\left(\sum_{l=1}^{s-k} 
\left(\frac{(l+k)!}{k!}\right)^{b\,p^*}\,\frac{y^l}{l!}\right)^{1/p^*}
\,=\,\left(\sum_{l=1}^{s-k}
\left((k+1)\cdots(l+k)\right)^{b\,p^*}\,\frac{y^l}{l!}\right)^{1/p^*}
\end{equation}
with
\[
 y\,=\,\frac{c_2^{p^*}}{(\tfrac{q}{p^*}+1)^{p^*/q}\,(a\,p^*-1)\,
  (k+1/2)^{a\,p^*-1}}. 
\]
Hence the upper bound in \eqref{pod-res} can be computed efficiently
using nested multiplication. We provide now the pseudo-code for doing
that:
\begin{eqnarray*}
&&y:=c_2^{p^*}/((q/p^*+1)^{p^*/q}(a\, p^*-1)(k+1/2)^{a\, p^*-1})\\
&&T:=y\, s^{b\, p^*}/(s-k)\\
&&\mbox{\bf for\ } l=s-k-1 \mbox{\ \bf to\ } 1 \mbox{\ \bf step\ } 
-1 \mbox{\ \bf do}\\
&&\mbox{}\qquad T:=(T+1)(l+k)^{b\, p^*}y /l\\
&&\mbox{\bf endfor}\\
&&T:=T^{1/p^*}.
\end{eqnarray*}

Furthermore, for $k\ge2$,
\begin{eqnarray*}
\lefteqn{\left(\sum_{\setv\subseteq[k]}\frac{\gamma_\setv^{p^*}}{(\tfrac{q}{p^*}+1)^{p^* |\setv|/q }}\right)^{1/p^*}}\\
& \le & c_1\left(1+c_2^{p^*} \sum_{j=1}^k\frac{j^{-a\,p^*}}{(q/p^*+1)^{p^*/q}}
  +\sum_{\substack{\setu\subseteq[k]\\ |\setu|\ge2}}\frac{(|\setu|!)^{b\,p^*}\,c_2^{|\setu|\, p^*}}{(q/p^*+1)^{p^*|\setu|/q}}
   \prod_{j\in\setu}j^{-a\,p^*}\right)^{1/p^*}.
\end{eqnarray*}
Now we provide an estimate for the last sum in this expression.
We have
\begin{eqnarray*}
\lefteqn{\sum_{\setu\subseteq[k]\atop |\setu|\ge2}\frac{(|\setu|!)^{b\,p^*}\,c_2^{|\setu|\, p^*}}{(q/p^*+1)^{p^*|\setu|/q}}
   \prod_{j\in\setu}j^{-a\,p^*}}\\
 &=&
 \sum_{\ell=2}^k \frac{(\ell !)^{b\,p^*}\,c_2^{\ell\, p^*}}{(q/p^*+1)^{p^*\ell/q}}\sum_{\setu\subseteq[k]\atop |\setu|=\ell}
 \prod_{j\in\setu}j^{-a\,p^*}\\
 &=&\sum_{\ell=2}^k \frac{(\ell !)^{b\,p^*}\,c_2^{\ell\, p^*}}{(q/p^*+1)^{p^*\ell/q}}
 \sum_{\substack{\setu\subseteq[k]\\ |\setu|=\ell\\ \{1\}\subseteq\setu}}
 \prod_{j\in\setu}j^{-a\,p^*} +
 \sum_{\ell=2}^{k-1} \frac{(\ell !)^{b\,p^*}\,c_2^{\ell\, p^*}}{(q/p^*+1)^{p^*\ell/q}}
 \sum_{\substack{\setu\subseteq[2:k]\\ |\setu|=\ell}}
 \prod_{j\in\setu}j^{-a\,p^*}\\
 &=&\sum_{\ell=2}^k \frac{(\ell !)^{b\,p^*}\,c_2^{\ell\, p^*}}{(q/p^*+1)^{p^*\ell/q}}
 \sum_{\substack{\setu\subseteq[2:k]\\ |\setu|=\ell-1}}
 \prod_{j\in\setu}j^{-a\,p^*} +
 \sum_{\ell=2}^{k-1} \frac{(\ell !)^{b\,p^*}\,c_2^{\ell\, p^*}}{(q/p^*+1)^{p^*\ell/q}}
 \sum_{\substack{\setu\subseteq[2:k]\\ |\setu|=\ell}}
 \prod_{j\in\setu}j^{-a\,p^*}.
 \end{eqnarray*}
For the two inner sums in the last expression we can now use the same 
method as in the derivation of \eqref{pod-res} for the terms with 
indices $\ell=2,\ldots,k-1$, and hence obtain
\begin{eqnarray}\label{bad}
\lefteqn{\sum_{\setu\subseteq[k]\atop |\setu|\ge2}\frac{(|\setu|!)^{b\,p^*}\,c_2^{|\setu|\, p^*}}{(q/p^*+1)^{p^*|\setu|/q}}
   \prod_{j\in\setu}j^{-a\,p^*}}\nonumber\\ 
   &\le&\sum_{\ell=2}^{k-1} (\ell !)^{b\,p^*-1}\left(\frac{c_2^{p^*}}{(q/p^*+1)^{p^*/q}(ap^*-1)1.5^{ap^*-1}}\right)^\ell
   (\ell (ap^*-1)1.5^{ap^*-1} +1)\nonumber\\
   &&+\frac{ (k!)^{b\,p^*-a\,p^*} c_2^{k p^*}}{(q/p^*+1)^{p^* k/q}}.
\end{eqnarray}

\end{example}

In analogy to product weights and using the upper bounds above, we calculated numbers 
$k=k(\varepsilon)$ which 
guarantee that
\[
 \left(\sum_{\setu\not\subseteq[k]}\frac{\gamma_\setu^{p^*}}
{(\tfrac{q}{p^*}+1)^{p^* |\setu|/q}}\right)^{1/p^*}\le \frac{\varepsilon}{2^{1/p^*}}.
\]
Since the upper bound \eqref{bad} is not sharp for large $k$, i.e.,
small $\varepsilon$, we calculated the values of $k(\varepsilon)$ only for $a=4$. More
precisely we did it for $s=10000$, $b=c_1=c_2=1$, $q=2$, $p=2$ and $p=\infty$,
and $a=4$. 

\[
  \begin{array}{c||c|c|c|c|c|c||l}
  \varepsilon     & 10^{-1} & 10^{-2} & 10^{-3}&10^{-4} & 10^{-5} & 10^{-6}& \\
\hline
\hline
  k(\varepsilon)   & 3      & 8     &  26    & 81      & 256    & 809
  &p=\infty\\
\hline
  k(\varepsilon)   & 2      & 5     &  12   &  29  &  74     & 185
&p=2\end{array} 
\]


\noindent {\bf Author's Address}\\

\noindent Peter Kritzer, Johann Radon Institute for Computational and Applied Mathematics (RICAM), Austrian Academy of Sciences, Altenbergerstr. 69, 4040 Linz, Austria. \\
\noindent Email: \texttt{peter.kritzer(AT)oeaw.ac.at}\\
\\

\noindent Friedrich Pillichshammer, Department of Financial Mathematics and Applied Number Theory, Johannes Kepler University Linz, Altenbergerstr. 69, 4040 Linz, Austria. \\
\noindent Email: \texttt{friedrich.pillichshammer(AT)jku.at}\\
\\

\noindent G.W. Wasilkowski, Computer Science Department, University of Kentucky, 301 David Marksbury Building, Lexington, KY 40506, USA.\\
\noindent Email: \texttt{greg(AT)cs.uky.edu}

\end{document}